\documentclass[11pt]{article}
\usepackage{amsfonts,amsmath,bm,amssymb,xspace,amsthm}
\usepackage{multirow,graphicx, algorithm,algorithmic,rotating}
\usepackage{url,cite}
\usepackage{fullpage}
\usepackage[small,bf]{caption}
\newtheorem{theorem}{Theorem}[section]
\newtheorem{lemma}[theorem]{Lemma}
\newtheorem{proposition}[theorem]{Proposition}

\theoremstyle{definition}
\newtheorem{example}[subsection]{Example}

\renewenvironment{proof}{\paragraph{Proof:}}{\hfill$\square$}

\usepackage[numbers,sort,compress]{natbib}

\usepackage[utf8]{inputenc} %
\usepackage[T1]{fontenc}    %
\usepackage{hyperref}       %
\usepackage{booktabs}       %
\usepackage{amsfonts}       %
\usepackage{nicefrac}       %
\usepackage{microtype}      %

\usepackage{mathtools}
\usepackage{makecell}

\usepackage[color=gray!40,textsize=tiny]{todonotes}
\presetkeys{todonotes}{inline}{}

\usepackage[labelfont=bf]{subcaption}
\usepackage[labelfont=bf]{caption}

\usepackage[T1]{fontenc}
\usepackage{lmodern}

\newif\ifshort
\shortfalse

\newcommand{\minimize}{\mbox{minimize}}

\newcommand{\argmax}{\mbox{argmax}}
\newcommand{\st}{\mbox{subject to}}

\newcommand{\tf}[1]{\boldsymbol{#1}}

\newcommand{\trueA}{A }
\newcommand{\trueB}{B }

\newcommand{\tfxw}{\tf{\Phi}_{\tf x\tf w}}
\newcommand{\tfxn}{\tf{\Phi}_{\tf x\tf e}}
\newcommand{\tfuw}{\tf{\Phi}_{\tf u\tf w}}
\newcommand{\tfun}{\tf{\Phi}_{\tf u\tf e}}

\newcommand{\rowtf}{\left\{\tfxw,\tfxn,\tfuw,\tfun \right\}}

\newcommand{\fulltf}{\begin{bmatrix} \tfxw &\tfxn  \\ \tfuw &\tfun  \end{bmatrix}}

\newcommand{\vop}{\mbox{vec}}
\newcommand{\vbar}{\overline{\vop}}

\newcommand{\norm}[1]{\lVert #1 \rVert}
\newcommand{\bignorm}[1]{\left\lVert #1 \right\rVert}
\newcommand{\twonorm}[1]{\lVert #1 \rVert_{2}}

\newcommand{\E}{\mathbb{E}}

\newcommand{\statedim}{n}
\newcommand{\inputdim}{m}
\newcommand{\outputdim}{\ell}
\newcommand{\noisedim}{w}

\newcommand{\hinf}{\mathcal{H}_\infty}
\newcommand{\htwo}{\mathcal{H}_2}
\newcommand{\lone}{\mathcal{L}_1}

\newcommand{\iid}{\stackrel{\mathclap{\text{\scriptsize{ \tiny i.i.d.}}}}{\sim}}

\newcommand{\cvectwo}[2]{\begin{bmatrix} #1 \\ #2 \end{bmatrix}}
\newcommand{\rvectwo}[2]{\begin{bmatrix} #1 & #2 \end{bmatrix}}

\newcommand{\R}{\mathbb{R}}

\newcommand{\calA}{\mathcal{A}}

\newcommand{\calN}{\mathcal{N}}
\newcommand{\calX}{\mathcal{X}}

\newcommand{\calS}{\mathcal{S}}

\newcommand{\rvecthree}[3]{\begin{bmatrix} #1 & #2 & #3 \end{bmatrix}}

\newcommand{\eps}{\varepsilon}

\title{Robust Guarantees for Perception-Based Control}

\author{Sarah Dean$^1$, Nikolai Matni$^2$, Benjamin Recht$^1$, and Vickie Ye$^1$\\
$^1$Department of EECS,  University of California, Berkeley\\
$^2$Department of ESE,  University of Pennsylvania\\}

\date{December 2019}

\begin{document}

\maketitle
\vspace{-1em}
\begin{abstract}%
Motivated by vision-based control of autonomous vehicles, we consider the problem of controlling a known linear dynamical system for which partial state information, such as vehicle position, is extracted from complex and nonlinear data, such as a camera image.
Our approach is to use a learned perception map that predicts some linear function of the state and to design a corresponding safe set and robust controller for the closed loop system with this sensing scheme. 
We show that under suitable smoothness assumptions on both the perception map and the generative model relating state to complex and nonlinear data,
parameters of the safe set can be learned via appropriately dense sampling of the state space.
We then prove that the resulting perception-control loop has favorable generalization properties.
We illustrate the usefulness of our approach on a synthetic example and on the self-driving car simulation platform CARLA.
\end{abstract}


\section{Introduction}

Incorporating insights from rich, perceptual sensing modalities such as cameras remains a major challenge in controlling complex autonomous systems.
While such sensing systems clearly have the potential to convey more information than simple, single output sensor devices, interpreting and robustly acting upon the high-dimensional data streams remains difficult.
Recent end-to-end approaches tackle the problem of image based control by learning an optimized map from pixel values directly to low level control inputs.
Though there has been tremendous success in accomplishing sophisticated tasks, critical gaps in understanding robustness and safety still remain \citep{levine2016end}.
On the other hand, methods rooted in classical state estimation and robust control explicitly characterize a model of the underlying system and its environment in order to design a feedback controller.
These approaches have provided strong and rigorous guarantees of robustness and safety in domains such as aerospace and process control, but the level of specification of the underlying system required has thus far limited their impact for complex sensor inputs.

In this paper, we aim to leverage contemporary techniques from machine learning and robust control to understand the conditions under which safe and reliable behavior can be achieved in controlling a system with uncertain and high-dimensional sensors.
Whereas much recent work has been devoted to proving safety and performance guarantees for learning-based controllers applied to systems with unknown dynamics~\citep{wabersich2018linear,akametalu2014reachability,ostafew2014learning,hewing2017cautious,dean2017sample,dean2018regret,abbasi2019model,abbasi2011regret,cheng2019end,taylor2019episodic,williams2018information,agarwal2019online,berkenkamp2017safe},
we focus on the practical scenario where the underlying dynamics of a system are well understood, and it is instead the integration of  perceptual sensor data into the control loop that must be learned.
Specifically, we consider controlling a known linear dynamical system for which partial state information can only be extracted from complex observations. 
Our approach is to design a \emph{virtual sensor} by learning both a perception map (i.e., a map from observations to a linear function of the state) and a bound on its estimation errors. 
We show that under suitable smoothness assumptions, we can guarantee bounded errors within a neighborhood of the training data. 
This 
model of uncertainty allows us to synthesize a robust controller that ensures that the system does not deviate too far from states visited during training. 
Finally, we show that the resulting perception and robust control loop is able to robustly generalize under adversarial noise models. 
To the best of our knowledge, this is the first such guarantee for a vision-based control system.

\ifshort
\else
\subsection{Related work}
\paragraph{Vision based estimation, planning, and control} There is a rich body of work, spanning several research communities, that integrate complex sensing modalities, specifically cameras, into estimation, planning, and control loops. 
The robotics community has focused mainly on integrating camera measurements with inertial odometry via an Extended Kalman Filter (EKF)~\citep{jones2011visual,kelly2011visual,hesch2014camera}. 
Similar approaches have also been used as part of Simultaneous Localization and Mapping (SLAM) algorithms in both ground ~\citep{lynen2015get} and aerial~\citep{lynen2013robust} vehicles. 
We note that these works focus solely on the estimation component, and do not consider downstream use of state estimates in control loops. 
In contrast, the papers ~\citep{loianno2016estimation,tang2018aggressive,lin2018autonomous} demonstrate techniques that use camera measurements to aid inertial position estimates to enable aggressive control maneuvers in unmanned aerial vehicles.  

The machine learning community has taken a more data-driven approach.
The earliest such example is likely~\citep{pomerleau1989alvinn}, in which a 3-layer neural-network is trained to infer road direction from images. 
Modern approaches to vision based planning, typically relying on deep neural networks, include learning maps from image to trail direction ~\citep{giusti2015machine}, learning Q-functions for indoor navigation using 3D CAD images~\citep{sadeghi2016cad2rl}, and using images to specify waypoints for indoor robotic navigation~\citep{bansal2019combining}. 
Moving from planning to low-level control, end-to-end learning for vision based control has been achieved through imitation learning from training data generated via human~\citep{bojarski2016end} and model predictive control~\citep{pan2018agile}. 
The resulting policies map raw image data directly to low-level control tasks.
In \citet{codevilla2018end}, higher level navigational commands, images, and other sensor measurements are mapped to control actions via imitation learning. 
Similarly, \citet{williams2018information} and related works, image and inertial data is mapped to a cost landscape, that is then optimized via a path integral based sampling algorithm. 
More closely related to our approach is~\citet{lambert2018deep}, where a deep neural network is used to learn a map from image to system state -- we note that this perception module is naturally incorporated into our proposed pipeline. 
To the best of our knowledge, none of the aforementioned results provide safety or performance guarantees.

\paragraph{Learning, robustness, and control} Our theoretical contributions are similar in spirit to those of the online learning community, in that we provide generalization guarantees under adversarial noise models~\citep{anava2013online,anava2015online,kuznetsov2016time,hassibi2001h,yasini2018worst}. 
Similarly, \citet{agarwal2019online} shows that adaptive disturbance feedback control of a linear system under adversarial process noise achieves sublinear regret -- we note that this approach assumes full state information.
We also draw inspiration from recent work that seeks to bridge the gap between linear control and learning theory. 
These assume a linear time invariant system, and derive finite-time guarantees for system identification \citep{dean2017sample,hardt2018gradient,simchowitz2018learning,sarkar2018fast,oymak2018non,sarkar2019finite,tsiamis2019finite,fattahi2018data,fattahi2019learning,pereira2010learning}, and/or integrate learned models into control schemes with finite-time performance guarantees \citep{dean2017sample,abbasi2011regret,russo17,abeille18,ouyang17,abbasi2019model,dean2018regret,cohen2019learning,mania2019certainty,rantzer2018concentration}.  
\fi

\ifshort
\else
\subsection{Notation}
\fi

\label{sec:notation}
We use letters such as $x$ and
$A$ to denote vectors and matrices, and boldface letters such as $\tf x$ and $\tf \Phi$ to denote infinite horizon signals and linear convolution operators. Thus, for $\tf y = \tf \Phi \tf x$, we have by definition that $y_k = \sum_{t=0}^k \Phi_t x_{k-t}$.  We write $x_{0:t} = \{x_0, x_1, \dots, x_t\}$ for the history of signal $\tf x$ up to time $t$. For a function $x_k \mapsto f_k(x_k)$, we write $\tf f(\tf x)$ to denote the signal $\{f_k(x_k)\}_{k=0}^\infty$.  
\ifshort
\else

\fi
We overload the norm $\norm{\cdot}$ so that it applies equally to elements $x_k$, signals $\tf x$, and linear operators $\tf \Phi$. 
For any element norm, we define the signal norm as $\norm{\tf x} = \sup_k \norm{x_k}$
and the linear operator norm as $\norm{\tf \Phi} = \sup_{\norm{\tf w} \leq 1} \norm{\tf \Phi \tf w}$. 
We primarily focus on the triple $(\norm{x_k}_\infty, \norm{\tf x}_\infty, \norm{\tf \Phi}_{\mathcal L_1})$. Note that as $\norm{\tf \Phi}$ is an induced norm, it satisfies the sub-multiplicative property $\norm{\tf \Phi \tf \Psi}\leq \norm{\tf \Phi}\norm{\tf \Psi}$. 
\ifshort
\else

\fi
We define the ball $B_r(x_0) = \{x~:~\|x-x_0\|\leq r\}$. 
We say that a fuction $f$ is locally $S$-slope bounded for a radius $r$ around $x_0$ if for
$x \in B_r(x_0)$, $\| f(x) - f(x_0) \| \le S \|x - x_0 \|$, or similarly locally $L$-Lipschitz if for $x,y \in B_r(x_0)$, $\norm{f(x) - f(y)} \le L \norm{x - y}$.

\section{Problem setting}
\label{sec:setup}
Consider the LTI dynamical system
\begin{align}
x_{k+1} &= A x_k + B u_k + Hw_k\:, \label{eq:dynamics}\\
z_k &= q(x_k)\label{eq:generative-model}
\end{align}
with system state $x \in \R^\statedim$, control input $u \in \R^\inputdim$, disturbance $w \in \R^\noisedim$, and observation $z\in\R^M$. We take the dynamics matrices $(A,B,H)$ to be known. 
The observation process is determined by the unknown \emph{generative model} $q$, which is nonlinear and potentially quite high-dimensional.
As an example, consider a camera affixed to the dashboard of a car tasked with driving along a road. 
Here, the observations $\{z_k\}$ are the captured images and the map $q$ generates these images as a function of position and velocity.

Motivated by such vision based control systems, our goal is to solve the optimal control problem
\begin{equation}
\begin{array}{rl}
\displaystyle\minimize_{\{\gamma_k\}}&c(\tf x, \tf u)~~\text{subject to} ~~\text{dynamics~\eqref{eq:dynamics} and measurement~\eqref{eq:generative-model}}, \ u_k = \gamma_k(z_{0:k}),
\end{array}
\label{eq:nonlinear-control}
\end{equation}
where $c(\tf x, \tf u)$ is a suitably chosen cost function 
and $\gamma_k$ is a measurable function of the image history $z_{0:k}$.
This problem is made challenging by the  nonlinear, high-dimensional, and unknown generative model.\footnote{
	We remark that a nondeterministic appearance map $q$ may be of interest for modeling phenomena like noise and environmental uncertainty. While we focus our exposition on the deterministic case, we note that many of our results can be extended in a straightforward manner to any noise class for which the perception map has a bounded response. 
	For example, many computer vision algorithms are robust to random Gaussian pixel noise, gamma corrections, or sparse scene occlusions.
	} 

\begin{figure}
\centering
\includegraphics[clip, trim=0cm 11.75cm 0cm 3cm, height=0.15\textheight]{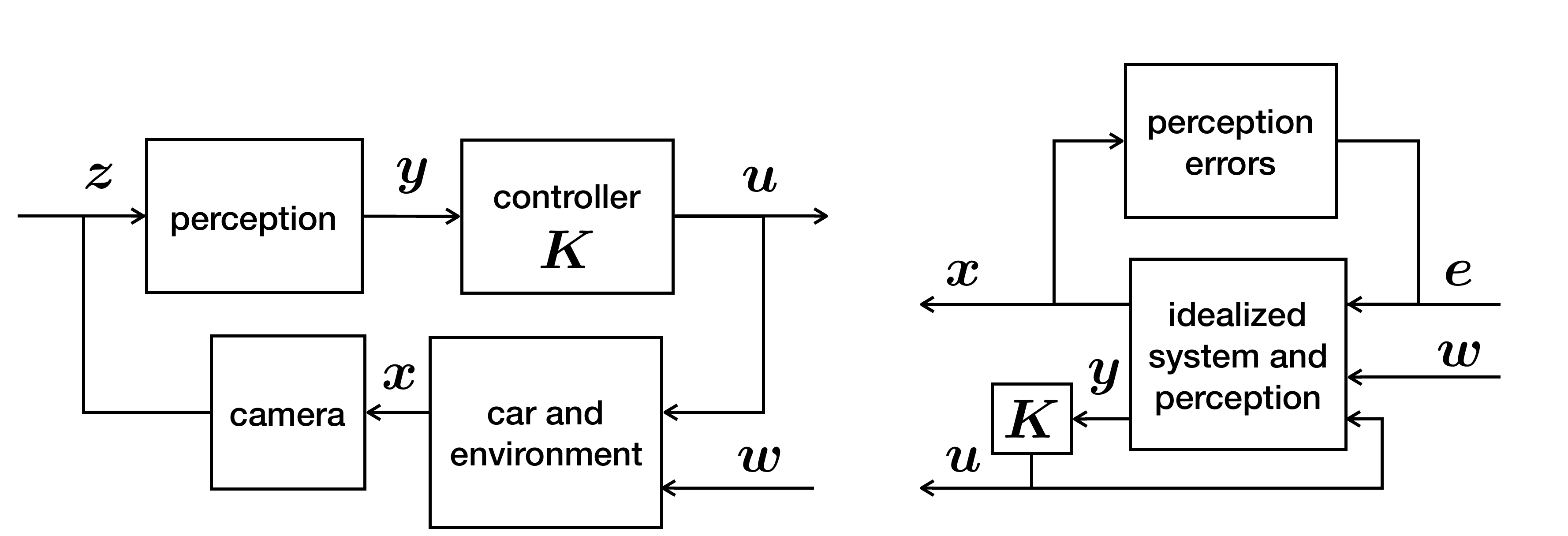}
\caption{(Left) A diagram of the proposed perception-based control pipeline.
    (Right) The conceptual rearrangement of the closed-loop system permitted through our perception error characterization. 
}\label{fig:pipeline}
\end{figure}
Suppose instead that there exists a \emph{perception map} $p$ that imperfectly predicts partial state information; that is
$p(z_k) = Cx_k + e_k$ for $C \in \R^{\outputdim \times \statedim}$ a known matrix,
and  error $e_k \in \R^\outputdim$.
Using this map, we define a new measurement model in which the map $p$ plays the role of a noisy sensor:
\begin{equation}
y_k = p(z_k) = Cx_k + e_k.
\label{eq:output-measurement}
\end{equation}  
This allows us to reformulate problem \eqref{eq:nonlinear-control} as a \emph{linear} optimal control problem, where the measurements are defined by~\eqref{eq:output-measurement} and the control law $u_k = \tf K(y_{0:k})$
is a \emph{linear}
function of the outputs of past measurements $y_{0:k}$. 
As illustrated in Figure~\ref{fig:pipeline}, 
guarantees of performance, safety, and robustness require designing a controller which suitably responds to system  disturbance $\tf w$ and sensor noise $\tf e$. 
\ifshort
\else
	For linear optimal control problems, a variety of cost functions and noise models have well-understood solutions; we expand on this observation further in Section \ref{sec:optimal_control_background}.
\fi

In light of this discussion, we approach our problem with the following data-driven procedure:
First, we use supervised learning to fit a perception map and explicitly characterize its errors in terms of the data it was trained on.
Second, we compute a robust controller that mitigates the effects of the measurement error $\tf e$.
We will show that under suitable local smoothness assumptions on the generative model $q$ and perception map $p$ we can guarantee closed-loop robustness using a number of samples depending only on the state dimension.

\ifshort
\else

\subsection{Background on linear optimal control}
\label{sec:optimal_control_background}

We first recall some basic concepts from linear optimal control in the partially observed setting. 
In particular we consider the optimal control problem
\begin{equation}
\begin{array}{rl}
\minimize_{\tf K} & c(\tf x, \tf u)\\
\st & x_{k+1} = Ax_k + Bu_k + Hw_k\\
& y_k = Cx_k + e_k\\
& u_k = \tf K(y_{0:k}),
\end{array}
\label{eq:of-sys}
\end{equation}
for $x_k$ the state, $u_k$ the control input, $w_k$ the process noise, $e_k$ the sensor noise, $\tf K$ a linear-time-invariant operator, and $c(\tf x, \tf u)$ a suitable cost function.

Control design depends on how the disturbance $\tf w$ and measurement error $\tf e$ are modeled, as well as performance objectives.
By modeling the disturbance $\tf w$ and sensor noise $\tf e$ as being drawn from different signal spaces, and by choosing correspondingly suitable cost functions, we can incorporate practical performance, safety, and robustness considerations into the design process. 
For example, the well-studied LQG control problem minimizes the expected quadratic cost 
for user-specified positive definite matrices and normally distributed process and measurement noise. 
The resulting optimal estimation and control policy is to apply a static state feedback controller with the Kalman filter state estimate.
The robust version of this problem considers instead the worst-case quadratic cost for $\ell_2$ or power-norm bounded noise. This is the $\mathcal{H}_\infty$ optimal control problem, which has a rich history \citep{zhou1996robust}.  

As a final example, we recall the $\lone$ optimal control problem, which seeks to minimize the $\ell_\infty \to \ell_\infty$ induced norm of a system.  
This formulation best accommodates real-time safety constraints and actuator saturation, which we motivate in the following example.

\begin{example}[Waypoint Tracking] \label{ex:ref_tracking}
Suppose that the control objective is to keep the system within a bounded distance of a sequence of waypoints at every timestep, and that the distance between sequential waypoints is bounded. Denoting the system state as $\tf\xi$ and the waypoint sequence as $\tf r$, the
 cost function is
\[
c(\tf \xi, \tf u) = \sup_{\substack{\norm{r_{k+1}-r_{k}}_\infty\leq 1,\\
\|e_k\|_\infty\leq 1, k\geq 0}}\bignorm{\begin{matrix}Q^{1/2}(\xi_k-r_k) \\ R^{1/2}u_k\end{matrix}}_\infty\:.
\]
If we specify costs by $Q = \mathrm{diag}{(1/q^2_{i})}$ and $R = \mathrm{diag}{(1/r^2_{i})}$, then as long as the optimal cost is less than 1, we can guarantee bounded tracking error $|\xi_{i,k} - r_{i,k}| \leq q_i$ and actuation $|u_{i,k}| \leq r_i$ for all possible realizations of the waypoint and sensor error processes. 
By defining an augmented state, i.e. $x_k := [\xi_k-r_k; r_k]$, and 
modeling waypoints as bounded disturbances, i.e.
$w_k:=r_{k+1}-r_{k}$, we can reformulate this cost into a standard $\mathcal{L}_1$ optimal control problem. 
\end{example}
In Appendix~\ref{app:optimal_control}, we provide further details on these optimal control problem formulations and solutions.
Returning to the goals of perception-based control, it is
clear that the perception errors $e_k$ will not follow a Gaussian distribution.
Because this invalidates the optimality of the LQG control strategy, we focus on bounded norm assumptions, particularly for the case of $\lone$ control.

 \fi

\section{Data-dependent perception error}
\label{sec:perception-error}

In this section, we introduce a procedure to estimate regions for which a learned perception map can be used safely during operation.
We first suppose access to initial training data $\calS_{0}=\{(x_i,z_i)\}_{i=1}^{N_0}$,
used to learn a perception map via any wide variety of traditional supervised methods.
We then estimate safe regions around the training data, potentially using parameters learned from a second dataset.
The result is a characterization of how quickly the learned perception map degrades as we move away from the initial training data.

We will describe the regions of the state space within which the sensing is reliable using a safe set which approximates sub-level sets of the error function $e(x) = p(q(x)) - Cx$.
We make this precise in the following Lemma, defining a safe set which is valid under an assumption that the error function is locally slope bounded around training data. 

\begin{lemma}[Closeness implies generalization] \label{lem:close_implies_gen}
Suppose that $p\circ q - C$ is locally $S$-slope bounded with a radius of $r$ around training datapoints.
Define the safe set
\begin{align}
\label{eq:safe_set}
\begin{split}
\calX_\gamma = \bigcup\limits_{(x_d,z_d)\in\calS_0}\{x\in B_r(x_d)~:~  ~~ \|p(z_d) - C x_d\| + S\|x-x_d\|\leq \gamma\}\:.
\end{split}
\end{align}
Then for any $(x, z)$ with $x\in\calX_\gamma$, the perception error is bounded:
$\norm{p(z) - Cx} \le \gamma$.
\end{lemma}

\ifshort
\else
\begin{proof}
The proof follows by a simple argument.
For training data point $(x_d, z_d)$ such that $x_{d}\in B_r(x)$,
\begin{align*}
    \norm{p(z) - Cx} &= \norm{p(q( x)) - C x - (p(q( x_d)) - Cx_d) + p(q( x_d)) - Cx_d} \\
    &\le S \norm{ x -  x_d} + \norm{p(q( x_d)) - C x_d}\:.
\end{align*}
The second line follows from the local slope bounded assumption.
Then further choosing $x_d$ to correspond to minimizing over training data in the definition of $\calX_\gamma$, we have
\begin{align*}
    \norm{p( z) - Cx} \le S \norm{x -  x_d} + \norm{p(q( x_d)) - C x_d} \leq \gamma\:.
\end{align*}
\end{proof}
\fi
The validity of the safe set $\calX_\gamma$ depends on bounding the slope of the error function locally around the training data. 
We remark that this notion of slope boundedness has connections to sector bounded nonlinearities, a classic setting for nonlinear system stability analysis~\citep{desoer1975feedback}.

Deriving the slope boundedness of the error function relies on the learned perception map as well as the underlying generative model, so we propose a second learning step to estimate a bound on $S$.
We use an additional dataset composed of samples around each training data point:
$\calS_N=\cup_{x_d\in\calS_0} \calS_{x_d}$,
where each $\calS_{x_d}$ contains $N$ points densely sampled from $B_r(x_d)$.
For each $x_d$ in the training set, we then fix a radius $r$ and estimate $S$ with the maximum
observed slope:
\begin{equation}
\widehat{S}_{x_d} = \max_{x_i \in \calS_{x_d}} s(x_i, x_d)\:,
\quad \textrm{where}~~
s(x, x') = \norm{e(x) - e(x')}/\norm{x - x'}.
\end{equation}
By taking the maximum over training datapoints $x_d$, the slope bound $S$ can be estimated with $|\calS_0|N$ samples.

\begin{proposition}\label{prop:slope_est}
Assume that the error function $p\circ q - C$ is locally $L$-Lipschitz for a radius $r$ 
around a training point $x_d$. 
Further assume that the true maximum slope $S_{x_d} := \max_{x \in B_r(x_d)} s(x, x_d)$ is achieved by a point with distance from $x_d$ greater than $0<\tau\leq r$.
Then for $\widehat S_{x_d}$ estimated with an $\eps$-covering\footnote{We can achieve an $\eps$-cover by gridding with $N=\left(\frac{r}{\eps}\right)^d$
points from $B_r(x_d)$.} of
$B_r(x_d)$,
\[S_{x_d} \leq \widehat{S}_{x_d}(1 + \eps/\tau) + L\eps/\tau \:.\]
\end{proposition}
\ifshort
\else
\begin{proof}
Let $e(x) = p(q(x)) - Cx$. For $x_* \in \argmax_{x \in B_r(x_d)} s(x, x_d)$, we have
\[
S = \frac{\|e(x_*) - e(x_d) \|}{\|x_* - x_d\|}
= \frac{\|e(x') - e(x_d) + e(x_*) - e(x')\|}{\|x_* - x_d\|}
\le \frac{s(x', x_d)\|x' - x_d\| + L\|x_* - x'\|}{\|x_* - x_d\|},
\]
for $x'$ in the $\eps$ cover of $B_r(x_d)$.
Because $S$ can be achieved at least $\tau$ away from $x_d$,
and $x_*$ is at most $\eps$ away from x', we can bound $S$ as
\[
S \le \frac{s(x', x_d)(\|x' - x_*\| + \|x_* - x_d\|) - L\|x_* - x'\|}{\|x_* - x_d\|}
\le \widehat{S}(1 + \eps/\tau) + L\eps/\tau.
\]

\end{proof}
\fi

This result shows that adding a margin to the observed slope bound
 $\widehat S$ provides an upper bound on $S$. 
This margin decreases with the number of samples $N$ and depends on an assumption about the local smoothness of the system. 
This estimation problem is similar to Lipschitz estimation, which has been widely studied with an eye towards optimization \citep{sergeyev2010lipschitz}.
We note that Lipschitz smoothness is a stronger condition than our bounded slope condition, and requires several additional smoothness assumptions
to estimate \citep{wood1996estimation}.

By combining Lemma~\ref{lem:close_implies_gen} with Proposition~\ref{prop:slope_est}, we estimate bounds on the perception errors within a neighborhood of the training data. In what follows, we use these local error bounds in robust control.

\section{Analysis and synthesis of perception-based controllers}

\label{sec:control}
The local generalization result in Lemma \ref{lem:close_implies_gen} is useful only if the system remains close to states visited during training.  
In this section, we show
that robust control can ensure that the system will remain close to training data so long as the perception map generalizes well. 
By then enforcing that the composition of the two bounds 
is a contraction, a natural notion of controller robustness emerges that guarantees favorable behavior and generalization. 
In what follows, we adopt an adversarial noise model and exploit the fact that we can design system behavior to bound how far the system deviates from states visited during training. 

Recall that for a state-observation pair $(x,z)$, the perception error, defined as $e := p(z) - Cx$, acts as additive noise to the measurement model $y=p(z)$.
While standard linear control techniques can handle uniformly bounded errors,
more care is necessary to further ensure that the system remains with a safe region of the state space, as determined by the training data. 
Through a suitable convex reformulation of the safe region, this goal can be addressed through receding horizon strategies (e.g.~\cite{wan2002robust,mayne2006robust}). 
While these methods are effective in practice, constructing terminal sets and ensuring a priori that feasible solutions exist is not an easy task. To make explicit connections between learning and control, we turn our analysis to a system level perspective on the closed-loop to characterize its sensitivity to noise. 

\ifshort
\else

\subsection{System-level parametrization}\label{app:sec:SLS}
The system level synthesis (SLS) framework, proposed by~\citet{SysLevelSyn1}, provides a parametrization of our problem that makes explicit the effects of errors $\tf e$ on system behavior.
Namely, for any controller that is a linear function of the history of system outputs, 
we can write the state and input as a convolution of the system noise and the closed-loop system responses $\{\Phi_{xw}(k), \Phi_{xe}(k), \Phi_{uw}(k), \Phi_{ue}(k)\}$,
\begin{align}\label{eq:phis}
\begin{bmatrix} x_k \\ u_k \end{bmatrix} = \sum_{t=1}^{k} \begin{bmatrix} \Phi_{xw}(t) & \Phi_{xe}(t) \\ \Phi_{uw}(t) & \Phi_{ue}(t)\end{bmatrix} \begin{bmatrix} Hw_{k-t} \\ e_{k-t}\end{bmatrix} \:.
\end{align}
As the equation~\eqref{eq:phis} is linear in the system response elements $\Phi$, convex constraints on state and input translate to convex constraints on the system response elements. 
\citet{SysLevelSyn1} show that for any elements $\{\Phi_{xw}(k), \Phi_{xe}(k), \Phi_{uw}(k), \Phi_{ue}(k)\}$ constrained to obey, for all $k \geq 1$,
\begin{align*}
 \Phi_{xw}(1) = I,~~ 
\begin{bmatrix}\Phi_{xw}(k+1) &\Phi_{xe}(k+1) \end{bmatrix} &= \trueA 
\begin{bmatrix}\Phi_{xw}(k) &\Phi_{xe}(k) \end{bmatrix} + \trueB \begin{bmatrix}\Phi_{uw}(k) &\Phi_{ue}(k) \end{bmatrix},\\
\begin{bmatrix}\Phi_{xw}(k+1) \\ \Phi_{uw}(k+1) \end{bmatrix}
 &=  \begin{bmatrix}\Phi_{xw}(k) \\ \Phi_{uw}(k) \end{bmatrix}\trueA + \begin{bmatrix}\Phi_{xe}(k+1) \\ \Phi_{ue}(k+1) \end{bmatrix} C \:,
\end{align*}
there exists a feedback controller that achieves the desired system responses~\eqref{eq:phis}, and therefore any optimal control problem over linear systems can be cast as a corresponding optimization problem over system response elements. 

To lessen the notational burden of working with these system response elements, we will work with their $z$ transforms.
This is particularly useful for keeping track of semi-infinite sequences, especially since convolutions in time can now be represented as multiplications, i.e.
\[\begin{bmatrix} \tf x \\ \tf u\end{bmatrix} = \fulltf \begin{bmatrix} H\tf w \\ \tf e\end{bmatrix}\:.\]
The corresponding control law $\tf u = \tf K \tf y$ is given by $\tf K = \tfun - \tfuw \tfxw^{-1} \tfxn$.  This controller can be implemented via a state-space realization \citep{anderson2017structured} or as an interconnection of the system response elements \citep{SysLevelSyn1}.
The affine \emph{realizability constraints} can be rewritten as
\begin{align} \label{eq:sls_constraints}
\begin{bmatrix} zI-A & -B \end{bmatrix} \tf\Phi = \begin{bmatrix} I & 0  \end{bmatrix} , ~~ \tf \Phi \begin{bmatrix} zI-A \\ -C \end{bmatrix}  = \begin{bmatrix} I \\ 0  \end{bmatrix} \iff \tf\Phi \in\calA\:,
\end{align}
where we denote the affine space defined by $(A,B,C)$ as $\calA$.

In this SLS framework, many robust control costs can be written as system norms,
\[c(\tf x, \tf u) = 
\sup_{\substack{\norm{\tf w}\leq \eps_w\\ \norm{\tf e}\leq \eps_e}}
\left\|\begin{bmatrix}  Q^{1/2} & \\ & R^{1/2} \end{bmatrix}  
\fulltf
\begin{bmatrix}H\tf w\\ \tf e \end{bmatrix}\right\| = \left\|\begin{bmatrix}  Q^{1/2} & \\ & R^{1/2} \end{bmatrix}  
\fulltf
\begin{bmatrix}\eps_w H\\ \eps_e I \end{bmatrix}
\right\|\:.\]
For LQG control, the objective is equivalent to a system $\htwo$ norm \citep{zhou1996robust}.
\paragraph{A comment on finite-dimensional realizations}
Although the constraints and objective function in this framework are infinite dimensional, two finite-dimensional approximations have been successfully applied.  The first consists of selecting an approximation horizon $T$, and enforcing that $\tf\Phi(T) = 0$ for some appropriately large $T$, which is always possible for systems that are controllable and observable.  When this is not possible, one can instead enforce bounds on the norm of $\tf\Phi(T)$ and use robustness arguments to show that the sub-optimality gap incurred by this finite dimensional approximation decays exponentially in the approximation horizon $T$ \citep{anderson2019,boczar2018finite,dean2017sample,matni2017scalable}.  Finally, in the interest of clarity, we always present the infinite horizon version of the optimization problems, with the understanding that finite horizon approximations are necessary in practice.

 \subsection{Robust control for generalization}
\fi

In what follows, we specialize controller design concerns to our perception-based setting, and develop further conditions on the closed-loop response $\tf \Phi$ to incorporate into the synthesis problem.
Because system level parametrization makes explicit the effects of perception errors $\tf e$ on system behavior,
we can transparently keep track of closed-loop behavior, as made precise in the following Lemma.

\begin{lemma}[Generalization implies closeness]\label{lem:gen_implies_close}
For a perception map $p$ with errors $\tf e=p(\tf z)-C\tf x$,
let the system responses $\rowtf$ lie in the affine space defined by dynamics $(A,B,C)$,
and let $\tf K$ be the associated controller. Then the state trajectory $\tf x$ achieved by the control law
$\tf u = \tf K p(\tf z)$ and driven by noise process $\tf w$, satisfies, for any target trajectory $\tf x_d$,
\begin{align}
\norm{\tf x - \tf x_d} \leq \norm{\widehat{\tf x} - \tf x_d} + \norm{\tfxn}\norm{\tf e}.
\label{eq:gen_gives_close}
\end{align}
where we define the \emph{nominal closeness} $\norm{\widehat{\tf x} - \tf x_d} = \norm{\tfxw H \tf w - \tf x_d} $ to be the deviation from the target trajectory in the absence of measurement errors.
\end{lemma}
\ifshort
\else
\begin{proof}
Notice that over the course of a trajectory, we have system outputs $\tf y = p(\tf z) = C\tf x + \tf e$.
Then recalling that the system responses are defined such that $\tf x = \tfxw H \tf w + \tfxn \tf e$, we have that
\begin{align*}
\norm{\tf x - \tf x_d} &= \norm{\tfxw H \tf w + \tfxn \tf e - \tf x_d}\leq \norm{\tfxw H \tf w - \tf x_d} + \norm{\tfxn}\norm{\tf e} \:.
\end{align*}
\end{proof}

\fi
The terms in bound \eqref{eq:gen_gives_close} capture different generalization properties.
The first is small if we plan to visit states during operation that are similar to those seen during training.
The second term is a measure of the robustness of our system to the error
$\tf e$.

We are now in a position to state the main result of the paper, which shows that under an additional robustness condition,
Lemmas \ref{lem:close_implies_gen} and \ref{lem:gen_implies_close} combine to define a control invariant set around the training
data within which we can bound the perception errors and consequently the performance.

\begin{theorem}
Let the assumptions of Lemmas \ref{lem:close_implies_gen} and \ref{lem:gen_implies_close} hold and suppose that the training error is bounded, i.e. $\|p(z_d)-Cx_d\|\leq  R_0$ for all $(x_d,z_d)\in\calS_0$. Then as long as
\begin{equation}
\norm{\tfxn}\leq \frac{ 1 - \frac{1}{r}\norm{\widehat{\tf x} - \tf x_d} }{  S + \frac{R_0}{r}}\:,
\label{eq:xetahat-bound}
\end{equation}
the perception errors remain bounded
\begin{equation}
\displaystyle\norm{p(\tf z)  - C\tf x} \leq 
\frac{\norm{\widehat{\tf x} - \tf x_d} + R_0}{1-S\norm{\tfxn}} := \gamma,
\label{eq:gen-bound}
\end{equation}
and the closed-loop trajectory lies within $\calX_\gamma$.
\label{thm:simple-generalization}
\end{theorem}
\ifshort
\else
\begin{proof}
Recall that as in the proof of Lemma~\ref{lem:close_implies_gen}, as long as $x_k\in B_r(x_{d,k})$ for all $k$ and some arbitrary $x_{d,k}\in\calS_0$,
\begin{align}\label{eq:e_norm}
\|e_k\| \leq S \| x_k - x_{d,k}\|+\|e_{d_k}\|\implies \|\tf e\| = \max_k \| e_k\| \leq S \|\tf x - \tf x_d\|+\|\tf e_d\|\:.
\end{align}
By assumption, $\|\tf e_d\| \leq R_0$. Substituting this expression into the result of Lemma~\ref{lem:gen_implies_close}, we see that as long as $\|x_k - x_{d,k}\|\leq r$ for all $k$,
\[\norm{\tf x - \tf x_d} \leq \norm{\widehat{\tf x} - \tf x_d} + \norm{\tfxn}(S \|\tf x - \tf x_d\|+\|\tf e_d\|) 
\iff 
\norm{\tf x - \tf x_d} \leq \frac{\norm{\widehat{\tf x} - \tf x_d} + \eps_\eta \norm{\tfxn}}{1-S\norm{\tfxn}}\:,\]
Next, to ensure that the the radius is bounded, first note that by norm definition,
$\max_k\|x_k - x_{d,k}\|_\infty \leq r$ if and only if $\|\tf x - \tf x_{d}\|\leq r$.
A sufficient condition for this is given by
\[\frac{\norm{\widehat{\tf x} - \tf x_d} + R_0 \norm{\tfxn}}{1-S\norm{\tfxn}}\leq r
\iff 
\norm{\tfxn}\leq \frac{ 1 - \frac{1}{r}\norm{\widehat{\tf x} - \tf x_d} }{ S + \frac{R_0}{ r}}\:,\]
and thus we arrive at the robustness condition. 
Since the bound on $\norm{\tf x - \tf x_d}$ is valid, we now use it to bound $\|\tf e\|$, starting with~\eqref{eq:e_norm} and rearranging,
\[\|\tf e\| \leq \frac{\norm{\widehat{\tf x} - \tf x_d} + R_0}{1-S\norm{\tfxn}}\:.\]
\end{proof}
\fi

Theorem \ref{thm:simple-generalization} shows that the bound \eqref{eq:xetahat-bound} should be used during controller synthesis to ensure generalization.  
Feasibility of the synthesis problem depends on the controllability and observability of the system $(A,B,C)$, which impose limits on how small $\norm{\tfxn}$ can be made to be, and on the planned deviation from training data as captured by the quantity $\norm{\widehat{\tf x} - \tf x_d}$.

\ifshort
\else
\subsection{Robust synthesis for waypoint tracking}
\fi
We now specialize to the task of waypoint tracking to simplify the term $\norm{\widehat{\tf x} - \tf x_d}$ and propose a robust control synthesis problem.
\ifshort
\else
Recall that
\fi
waypoint tracking can be encoded by defining the state $x_k := [ \xi_k-r_k; r_k]$ as the concatenation of tracking error and waypoint and the disturbance as the change in reference, $w_k := r_{k+1} - r_k$.
\begin{proposition}
For a reference tracking problem with $\|r_{k+1} - r_k\|\leq \Delta_\mathrm{ref}$ and a reference trajectory within a ball of radius $r_\mathrm{ref}$ from the training data,
\[\norm{\widehat{\tf x} - \tf x_d} \leq \Delta_\mathrm{ref}\|\begin{bmatrix}I&0\end{bmatrix} \tfxw H\| + r_\mathrm{ref}\:.\]
\end{proposition}
\ifshort
\else
\begin{proof}
For the training data, we can assume zero tracking errors, so that $\tf \xi_d = \tf r_d$. Then $x_{d,k} = [0; \xi_{d,k}]$.
Thus, we can rewrite $\norm{\widehat{\tf x} - \tf x_d} \leq \|\tf \xi - \tf r\| + \|\tf r - \tf\xi_d\| = \|\begin{bmatrix}I&0\end{bmatrix} \tfxw \tf w\| + \|\tf r - \tf\xi_d\|$. 
\end{proof}

\fi
Using this setting, we add the robustness condition in~\eqref{eq:xetahat-bound} to a control synthesis problem as
\begin{equation*}
\begin{array}{rl}
\minimize_{\tf \Phi} & c(\tf \Phi)\\
\st &\tf\Phi \in\calA,~~ (S+\frac{R_0}{r})\norm{\tfxn} + \frac{\Delta_\mathrm{ref}}{r}\|\begin{bmatrix}I&0\end{bmatrix} \tfxw H\| + \frac{r_\mathrm{ref}}{r}\leq  1  \:.
\end{array}
\end{equation*}
Notice the trade-off between the size of different parts of the system response. 
Because the responses must lie in an affine space, they cannot both become arbitrarily small. 
Therefore, the synthesis problem must trade off between sensitivity to measurement errors and tracking fidelity.
These trade-offs are mediated by the quality of the perception map, through its slope-boundedness and training error, and the ambition of the control task, through its deviation from training data and distances between waypoints.

\ifshort
\else

While the robustness constraint in~\eqref{eq:xetahat-bound} is enough to guarantee a finite cost,
it does not guarantee a small cost.
To achieve this, we
incorporate the error bound to
arrive at the following robust synthesis procedure:
\begin{equation*}
\begin{array}{rl}
\minimize_{\tf \Phi,\gamma} & \left\|\begin{bmatrix} Q^{1/2} & \\ & R^{1/2} \end{bmatrix}  
\fulltf
\begin{bmatrix}\Delta_\mathrm{ref} H\\\gamma I \end{bmatrix}
 \right\|\\
\st &\tf\Phi \in\calA,~~ (S+\frac{R_0}{r})\norm{\tfxn} + \frac{\Delta_\mathrm{ref}}{r}\|\begin{bmatrix}I&0\end{bmatrix} \tfxw H\| + \frac{r_\mathrm{ref}}{r}\leq  1  \\
&\Delta_\mathrm{ref}\|\begin{bmatrix}I&0\end{bmatrix} \tfxw H\| + r_\mathrm{ref} + R_0 \leq \gamma(1-S\norm{\tfxn})\:.
\end{array}
\end{equation*}

This is a convex program for fixed $\gamma$, so the full problem can then be approximately solved via one-dimensional search.

 \fi

\ifshort
\else
\subsection{Necessity of robustness}
\fi
Robust control is notoriously conservative, and our main result reslies heavily on small gain-like arguments.
Can the conservatism inherent in this approach be generally reduced? In this section, we answer in the negative by describing 
\ifshort
\else
a class of examples
\fi 
for which the robustness condition in Theorem~\ref{thm:simple-generalization} is necessary.
For simplicity, we specialize to the goal of regulating a system to the origin, and assume that $(0,z_0)$ is in the training set, with $p(z_0) = 0$.

\ifshort
\else

We consider the following optimal control problem
\begin{equation*}
\begin{array}{rl}
\minimize_{\tf \Phi} & c(\tf \Phi)~~\st ~~\tf\Phi \in\calA,~~\|\tfxn\| \leq \alpha\:.
\end{array}
\end{equation*}
and define $\bar{\tf\Phi}$ as a minimizing argument for the problem without the inequality constraint. 
For simplicity, we consider only systems in which the closed-loop is strictly proper and has a state-space realization. %

\begin{proposition}
Suppose that 
\[\|\bar{\tf\Phi}_{\tf{xn}}\|_{\lone} = 
\|\bar{\tf\Phi}_{\tf{xn}}(1)\|_{\infty\to\infty}\:.\]
Then there exists a differentiable error function with slope bound $S$
such that $x=0$ is an asymptotically stable equilibrium if and only if $\alpha < \frac{1}{S}$.
\end{proposition}
\begin{proof}
Sufficiency follows from our main analysis, or alternatively from a simple application of the small gain theorem.
The necessity follows by construction, with a combination of classic nonlinear instability arguments and by considering real stability radii.

Recall that $y_k = Cx_k + (p(q(x_k)) -Cx_k)$. Define $\zeta_k$ to represent the internal state of the controller. Then we can write the system's closed-loop behavior as (ignoring for now the effect of process noise, which does not affect the stability analysis)

\begin{align*}
\begin{bmatrix}x_{k+1} \\\zeta_{k+1}\end{bmatrix}
= A_\mathrm{CL} \begin{bmatrix}x_{k} \\\zeta_{k}\end{bmatrix} + B_\mathrm{CL} e(x_k) \:.
\end{align*}
Suppose that the error function $e(x)$ is differentiable at $x=0$ with derivative $J$. Then the stability of the origin depends on the linearized closed loop system $A_\mathrm{CL}(J) := A_\mathrm{CL}  + B_\mathrm{CL} JC_\mathrm{CL}$ where $C_\mathrm{CL} = [I~0]$ picks out the relevant component of the closed-loop state. If any eigenvalues of $A_\mathrm{CL}(J)$ lie outside of the unit disk, then $x=0$ is not an asymptotically stable equilibrium.

Switching gears,  recall that ${\tf \Phi}_{\tf x\tf e}(z) = C_\mathrm{CL}(zI-A_\mathrm{CL})^{-1}B_\mathrm{CL}$. 
We define
$S = 1/\|\bar{\tf\Phi}_{\tf{xe}}\|$ and notice that
if $\alpha \geq \frac{1}{S}$, then $\bar{\tf\Phi}$ is also the solution to the constrained problem.
By assumption, $\|\bar{\tf\Phi}_{\tf{xn}}\|_{\lone} = \|\bar{\tf\Phi}_{\tf{xn}}(1)\|_{\infty\to\infty}$.
Thus, there exist real $w,v$ such that $v^\top \bar{\tf\Phi}_{\tf{xn}}(1) w = \|\bar{\tf\Phi}_{\tf{xn}}\|$.
If we set $J = S wv^\top$, then $A_\mathrm{CL}(J)$ has an eigenvalue on the unit disk and $x=0$ is not an asymptotically stable equilibrium.
To see why this is true, 
recall the classic result for stability radii (e.g.
Corollary 4.5 in~\citet{hinrichsen1990real}) and
notice
\[\bar{\tf\Phi}_{\tf{xn}}(1)  = \sum_{k\geq 0} \Phi_{xn}(k) 
 = \sum_{k\geq 0} C_\mathrm{CL} A_\mathrm{CL}^k B_\mathrm{CL}
 = C_\mathrm{CL}(I-A_\mathrm{CL})^{-1} B_\mathrm{CL}  \:.\]
Thus, for any error function with derivative $J=S wv^\top$ at zero, the robustness condition is necessary as well as sufficient. One such error function is simply $e(x) = J x$, though many more complex forms exist.
\end{proof}

We now present a simple example in which this condition is satisfied, and construct the corresponding error function which results in an unstable system.
\paragraph{Example}
\fi
Consider the double integrator  
\begin{align}
x_{k+1} = \begin{bmatrix} 1 & dt \\ 0 & 1 \end{bmatrix} x_k + \begin{bmatrix}0\\1\end{bmatrix}u_k,\quad y_k = p(z_k) = x_k + e(x_k)\:, \label{eq:double_int}
\end{align}
and the control law resulting 
\ifshort
\else
from the $\htwo$ optimal control problem with identity weighting matrices along with the robustness constraint $\|\tfxn\| \leq \alpha$.
\fi
Notice that the solution to the unconstrained problem would be the familiar LQG combination of Kalman filtering and static feedback on the estimated state.
\ifshort
\else
We denote the optimal unconstrained system response by $\bar{\tf\Phi}$.

Consider an error function with slope bounded by $S=1/{\|\bar{\tf\Phi}_{\tf x\tf e}\|}$ and derivative at $x=0$ equal to
$$J = \begin{bmatrix}-1/{\|\bar{\tf\Phi}_{\tf x\tf e}\|} & 0 \\ -1/{\|\bar{\tf\Phi}_{\tf x\tf e}\|} & 0\end{bmatrix}\:,$$ 
Calculations confirm that the $\lone$ norm satisfies the relevant property and that the origin is not a stable fixed point if the synthesis constraint $\alpha \geq \frac{1}{S}$.  
\fi 
\section{Experiments}\label{sec:experiments}

We demonstrate our theoretical results with examples of control from pixels, using both simple synthetic images and complex graphics simulation.
The synthetic example uses generated $64\times 64$ pixel images of a moving blurry white circle on a black background;
the complex example uses $800 \times 600$ pixel dashboard camera images of a vehicle in the CARLA simulator platform \citep{carla}.
\ifshort
\else
\begin{figure}[tb]
\centering
\includegraphics[height=.1\textheight]{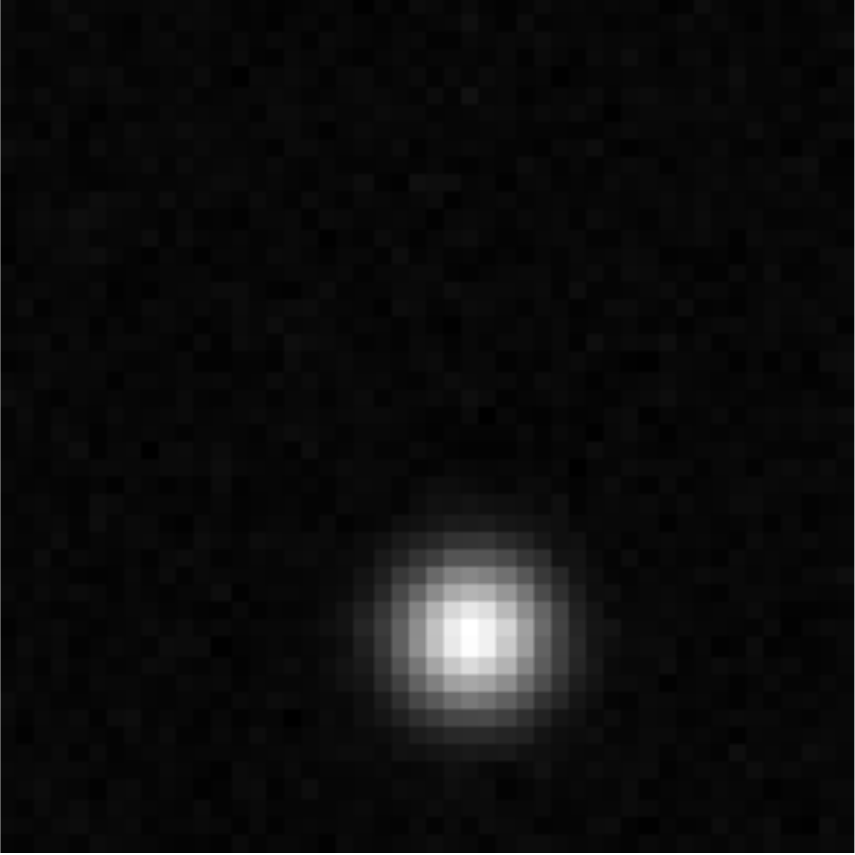}~~
\includegraphics[height=.1\textheight]{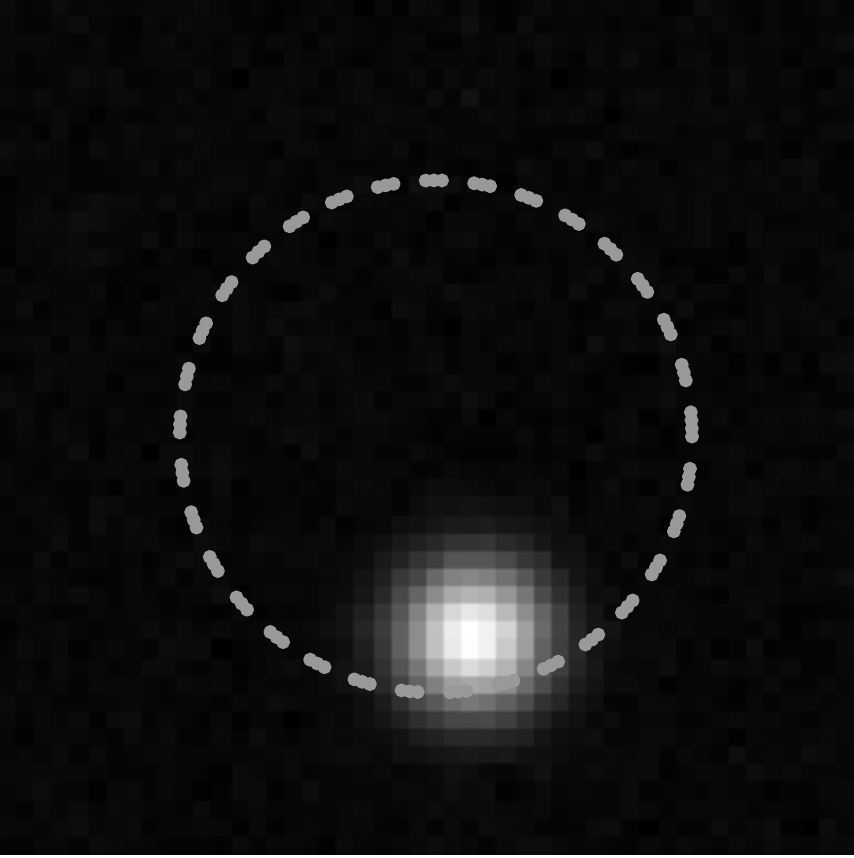}~~
\includegraphics[height=.1\textheight]{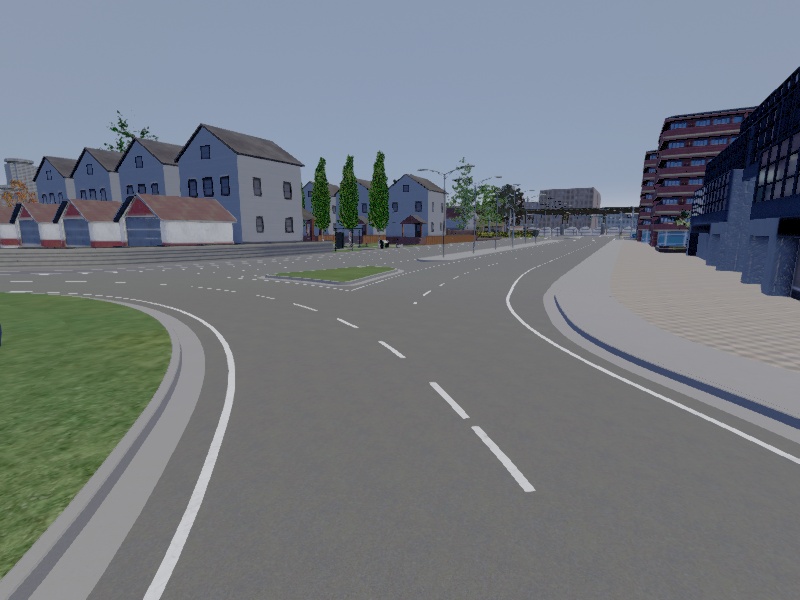}~~
\includegraphics[height=.1\textheight]{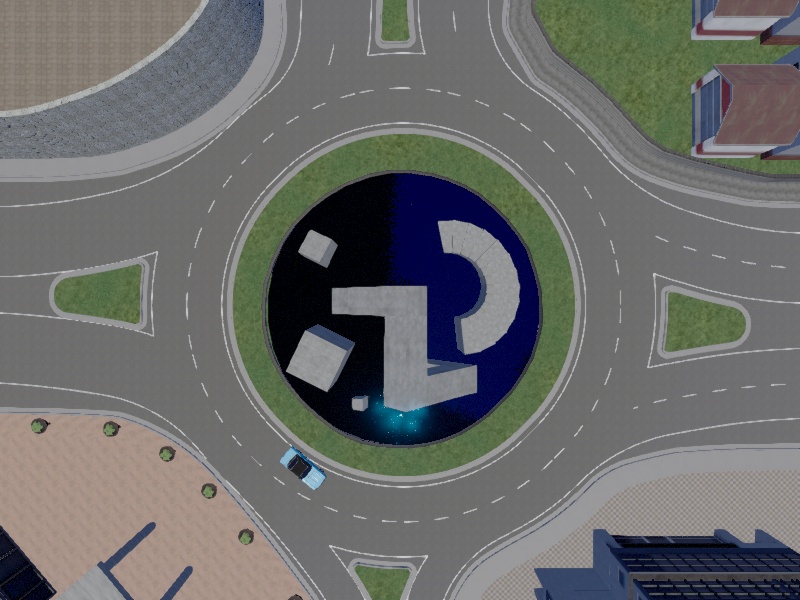}~~
\caption{In (a), visual inputs $\{z_t\}$ for the synthetic (left) and CARLA (right) examples.
In (b), (left) the $\ell_\infty$ bounded ball the synthetic circle must remain inside in the face of bounded adversarial noise,
(center) the nominal trajectory the synthetic circle is driven to follow, and (right) and the nominal trajectory the simulated vehicle is driven to follow.} \label{fig:visual_inputs_refs}
\end{figure}
Figure~\ref{fig:visual_inputs_refs}(a) shows representative images seen by the controllers.  
\fi

For both visual settings, we demonstrate our results in controlling a 2D double integrator to track a periodic reference.
The 2D double integrator state is given by $x_k^\top = [(x_k^{(1)})^\top\, (x_k^{(2)})^\top]$, and each component evolves independently according to the dynamics in~\eqref{eq:double_int} with $dt=0.1$.
For all examples, the sensing matrix $C$ extracts the position of the system, i.e., 
$Cx_k = [x^{(1)}_{1,k}, x^{(2)}_{1,k}]$.
Training and validation trajectories are generated by driving the system with an optimal state feedback controller (i.e. where measurement $\tf y = \tf x$)
to track a desired reference trajectory $\tf w = \tf r + \tf v$, 
where $\tf r$ is a nominal reference, and $\tf v$ is a random norm bounded random perturbation satisfying $\norm{v_k}_\infty\leq0.1$.

We consider three different perception maps: a linear map for the synthetic example and both visual odometry and a convolutional neural net for the CARLA example.
For the CNN, we collect a dense set of training examples around the reference to train the model.
We use the approach proposed by \citet{coates2012learning} to learn a convolutional representation of each image:
each resized and scaled image is passed through a single convolutional, ReLU activation, and max pooling layer.
We then fit a linear map of these learned image features to position and heading of the camera pose.
We require approximately $30000$ training points.
During operation, {pixel-data} $z$ is passed through the CNN to obtain {position estimates} $y$, {which are then used by the robust controller}.
We note that other more sophisticated architectures for feature extraction would also be reasonable
to use in our control framework; we find that this one is conceptually simple and sufficient for
tracking around our fixed reference.

To perform visual odometry, we first collect images from known poses around the desired
reference trajectory. We then use ORB SLAM \citep{orbslam} to build a global map
of visual features and a database of reference images with known poses.
This is the ``training'' phase.  We use one trajectory of $200$ points;
the reference database is approximately this size.
During operation, an image $z$ is matched with an image $z_d$ in the database.
The reprojection error between the matched features in $z_d$ with known pose $x_d$
and their corresponding features in $z$ is then minimized to generate pose estimate $y$.
For more details on standard visual odometry methods, see \cite{scaramuzza2011visual}.
We highlight that modern visual SLAM algorithms incorporate sophisticated
filtering and optimization techniques for localization in previously unseen environments
with complex dynamics; we use a simplified algorithm under this training and testing
paradigm in order to better isolate the data dependence.

We then estimate safe regions for all three maps, as described in Section~\ref{sec:perception-error}. %
In the top panels of Figure~\ref{fig:error-profile-rollouts} we show the error profiles as a function of distance
to the nearest training point for the linear (left), SLAM (middle), and CNN (right) maps.
We see that these data-dependent localization schemes exhibit the
local slope bounded property posited in the previous section.

\begin{figure}[tb]
\centering
\includegraphics[height=0.22\textheight]{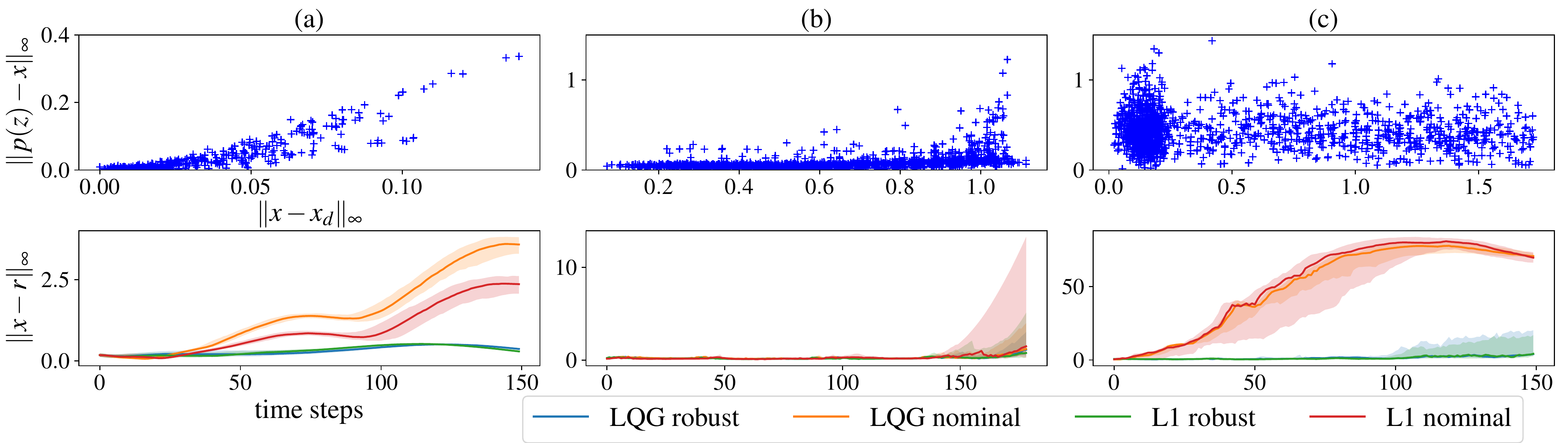}
\caption{
    (top) Test perception error $\|p(z) - Cx\|_\infty$ vs. distance to the nearest training point $\|x - x_d\|_\infty$;
    and (bottom) median $\ell_\infty$ tracking error for 200 rollouts using the (a) linear map on synthetic images,
	(b) SLAM and (c) CNN on CARLA dashboard images. Error bars show upper and lower quartiles.
}
\label{fig:error-profile-rollouts}
\end{figure}

We compare robust synthesis to the behavior of \emph{nominal controllers} that do not take into account the nonlinearity in the measurement model.
In particular, we compare the performance of naively synthesized LQG and $\mathcal L_1$ optimal controllers with controllers designed with the robustness condition~\eqref{eq:xetahat-bound}.
LQG is a standard control scheme that explicitly separates state estimation (Kalman Filtering) from control (LQR control), and is emblematic of much of standard control practice.
$\mathcal L_1$ optimal control minimizes worst case state deviation and control effort by modeling process and sensor errors as $\ell_\infty$ bounded adversarial processes.
For further discussion and details, refer to
\ifshort
\else
Section~\ref{sec:optimal_control_background} and Appendix~\ref{app:experiment-details}.
\fi
The bottom panels of Figure~\ref{fig:error-profile-rollouts} show that the robustly synthesized controllers remain within a bounded neighborhood around the training data.
On the other hand, the nominal controllers drive the system away from the training data, leading to a failure of the perception and control loop.
We note that although visual odometry may not satisfy smoothness assumptions 
when the feature detection and matching fails, we nevertheless observe safe system behavior, suggesting that using our robust controller, no such failures occur.

\section{Conclusion}\label{sec:conclusions}
Though standard practice is to treat the output of a perception module as an ordinary signal, we have demonstrated both in theory and experimentally that accounting for the inherent uncertainty of perception based sensors can dramatically improve the performance of the resulting control loop. Moreover, we have shown how to quantify and account for such uncertainties with tractable data-driven safety guarantees. We hope to extend this study to the control of more complex systems, and to apply this framework to standard model-predictive control pipelines which form the basis of much of contemporary control practice.

We further hope to highlight the challenges involved in adapting learning-theoretic notions of generalization to the setting of controller synthesis.
First note that if we collect data using one controller, and then use this data to build a new controller, there will be a distribution shift in the observations seen between the two controllers. Any statistical generalization bounds on performance must necessarily account for this shift. Second, from a more practical standpoint, most generalization bounds require knowing instance specific quantities governing properties of the class of functions we use to fit a predictor. Hence, they will include constants that are not measurable in practice. This issue can perhaps be mitigated using some sort of bootstrap technique for post-hoc validation. However, we note that the sort of bounds we aim to bootstrap are worst case, not average case. Indeed, the bootstrap  typically does not even provide a consistent estimate of the maximum of independent random variables, see for instance \cite{bickel1981some}, and Ch 9.3 in \cite{chernick2011bootstrap}. Other measures such as conditional value at risk require billions of samples to guarantee five 9s of reliability~ \citep{rockafellar2000optimization}. We highlight these issues only to point out that adapting statistical generalization to robust control is an active area with many open challenges to be considered in future work.

\subsubsection*{Acknowledgments}
This work is generously supported in part by ONR awards N00014-17-1-2191, N00014-17-1-2401, and N00014-18-1-2833, the DARPA Assured Autonomy (FA8750-18-C-0101) and Lagrange (W911NF-16-1-0552) programs, a Siemens Futuremakers Fellowship, and an Amazon AWS AI Research Award.
SD and VY are supported by an NSF Graduate Research Fellowship under Grant No. DGE 1752814.

{\small 
\bibliographystyle{plainnat}
\bibliography{references}}

\ifshort
\else
\newpage
\appendix

\section{Linear optimal control}\label{app:optimal_control}

\begin{table}[t]
\small
\centering
\begin{tabular}{|c|c|c|c|}
\cline{1-4}
\textbf{Name} & \textbf{Disturbance class} & \textbf{Cost function} & \textbf{Use cases}   \\ \cline{1-4}
 LQR/$\mathcal{H}_2$ & \makecell{$\E \nu = 0,$ \\ $\E \nu^4 < \infty$, $\nu_k$ i.i.d.}& $\displaystyle\E_\nu \left[\lim_{T\to\infty}\sum_{k=0}^T\frac{1}{T}x_k^\top Q x_k + u_k^\top R u_k \right] $& \makecell{Sensor noise, \\ aggregate behavior,\\ natural processes}   \\ \cline{1-4}
 $\mathcal{H}_\infty$ & $\norm{\nu}_{pow} \leq 1$ &  $\displaystyle\sup_{\norm{\nu}_{pow}\leq 1} \lim_{T\to\infty} \frac{1}{T}\sum_{k=0}^T  x_k^\top Q x_k + u_k^\top R u_k $ & \makecell{Modeling error,\\ energy/power \\constraints}   \\ \cline{1-4}
 $\mathcal{L}_1$ & $\norm{\nu}_\infty \leq 1$&\makecell{$ \displaystyle\sup_{\norm{\nu}_\infty\leq 1, k\geq 0}\bignorm{\begin{matrix}Q^{1/2}x_k \\ R^{1/2}u_k\end{matrix}}_\infty  $ \\} & \makecell{Real-time safety \\constraints,  actuator \\ saturation/limits}    \\ \cline{1-4}
\end{tabular}
\caption{Different noise model classes induce different cost functions, and can be used to model different phenomenon, or combinations thereof \citep{zhou1996robust,dahleh19871}.}
\label{tab:disturbances}
\end{table}

Table~\ref{tab:disturbances} shows several common cost functions that arise from different system desiderata and different classes of disturbances and measurement errors $\nu := (w,e)$. 
We recall the power norm,\footnote{The power-norm is a semi-norm on $\ell_\infty$, as $\norm{\tf x}_{pow} = 0$ for all $\tf x \in \ell_2$, and consequently is a norm on the quotient space $\ell_\infty/\ell_2$. It can be used to define the $\hinf$ system norm (see \cite{you2014h} and references therein).} defined as 
\[\norm{\tf x}_{pow} :=  \sqrt{ \lim_{T\to\infty}\frac{1}{T}\sum_{k=0}^T \twonorm{x_k}^2}\:.\]

We now recall some examples from linear optimal control in the partially observed setting. 

\begin{example}[Linear Quadratic Regulator]\label{ex:LQR}
Suppose that the cost function is given by 
\[
c(\tf x, \tf u) = \mathbb E_\nu \left[ \lim_{T\to \infty} \frac{1}{T} \sum_{k=0}^T x_k^\top Q x_k + u_k^\top R u_k \right],
\]
for some user-specified positive definite matrices $Q$ and $R$, $w_k \iid  \calN(0,I)$, $H=I$, and that the controller is given full information about the system, i.e., that $C=I$ and $e_k = 0$ such that the measurement model collapses to $y_k = x_k$.  Then the optimal control problem reduces to the familiar Linear Quadratic Regulator (LQR) problem
\begin{equation}
\begin{array}{rl}
\displaystyle\minimize_{\{\pi\}}&\E_w \left[\displaystyle\lim_{T\to\infty}\frac{1}{T}\sum_{k=0}^T  x_k^\top Q x_k + u_k^\top R u_k \right]\\
\mathrm{subject to} & x_{k+1} = A x_k + B u_k + w_k \\
& u_k = \pi(x_{0:k}),
\end{array}
\label{eq:lqr-control}
\end{equation}
For stabilizable $(A,B)$, and detectable $(A,Q)$, this problem has a closed-form stabilizing controller based on the solution of the discrete algebraic Riccati equation (DARE)~\citep{zhou1996robust}.
This optimal control policy is linear, and given by
\begin{equation}
u^{\mathrm{LQR}}_k = -(B^\top P B + R)^{-1} B^\top P A x_k =: K_{\mathrm{LQR}} x_k,
\end{equation}
where $P$ is the positive-definite solution to the DARE defined by $(A,B,Q,R)$.
\end{example}

\begin{example}[Linear Quadratic Gaussian Control]\label{ex:LQG}
Suppose that we have the same setup as the previous example, but that now the measurement is instead given by \eqref{eq:output-measurement} for some $C$ such that the pair $(A,C)$ is detectable, and that $e_k \iid \calN (0,I)$. 
Then the optimal control problem reduces to the Linear Quadratic Gaussian (LQG) control problem, the solution to which is:
\begin{equation}
u^{\mathrm{LQG}}_k = K_{\mathrm{LQR}} \hat{x}_k,
\end{equation}
where $\hat{x}_k$ is the Kalman filter estimate of the state at time $k$. The steady state update rule for the state estimate is given by
\[\hat x_{k+1} = A\hat x_k + B u_k + L_{\mathrm{LQG}}(y_{k+1} - C(A\widehat x_k + Bu_k))\]
for filter gain $L_{\mathrm{LQG}} = -PC^\top (C P C^\top + I)^{-1} $
where $P$ is the solution to the DARE defined by $(A^\top,C^\top,I,I)$.
This optimal output feedback controller satisfies the 
\emph{separation principle}, meaning that the optimal controller $K_{\mathrm{LQR}}$ is computed independently of the optimal estimator gain $L_{\mathrm{LQG}}$.
\end{example}

These first two examples are widely known due to the elegance of their closed-form solutions and the simplicity of implementing the optimal controllers. 
However, this optimality rests on stringent assumptions about the distribution of the disturbance and the measurement noise. We now turn to an example, first introduced in Example~\ref{ex:ref_tracking}, for which disturbances are adversarial and the separation principle fails.

Consider a waypoint tracking problem where it is known that both the distances between waypoints and sensor errors are instantaneously $\ell_\infty$ bounded, and we want to ensure that the system remains within a bounded distance of the waypoints.  In this setup, the $\mathcal{L}_1$ optimal control problem is most natural, and our cost function is then
\[
c(\tf x, \tf u) = \sup_{\substack{\norm{r_{k+1}-r_{k}}_\infty\leq 1,\\
\|e_k\|_\infty\leq 1, k\geq 0}}\bignorm{\begin{matrix}Q^{1/2}(x_k-r_k) \\ R^{1/2}u_k\end{matrix}}_\infty,
\]
for some user-specified positive definite matrices $Q = \mathrm{diag}{\frac{1}{q^2_{i}}}$ and $R = \mathrm{diag}{\frac{1}{r^2_{i}}}$.  Then if the optimal cost is less than 1, we can guarantee that $|x_{i,k} - r_{i,k}| \leq q_i$ and $|u_{i,k}| \leq r_i$ for all possible realizations of the waypoint and sensor error processes. 
Considering the one-step lookahead case,\footnote{
	A similar formulation exists for any $T$-step lookahead of the reference trajectory.
} we can define
the augmented state $\xi_k = [x_k-r_k; r_k]$ %
and pose the problem with bounded disturbances
$w_k=r_{k+1}-r_{k}$. We can then formulate the following $\mathcal{L}_1$ optimal control problem,
\begin{equation}
\begin{array}{rl}
\displaystyle\minimize_{\{\pi\}}&\sup_{\norm{\nu}_\infty\leq 1, k\geq 0}\bignorm{\begin{matrix}\bar Q^{1/2}\xi_k \\ R^{1/2}u_k\end{matrix}}_\infty\\
\text{subject to} & \xi_{k+1} = \bar A \xi_k + \bar B u_k + \bar H w_k\:, ~~y_k = \bar C \xi_k + \eta_k\\
& u_k = \pi(y_{0:k}),
\end{array}
\label{eq:l1-control}
\end{equation}
where
	\[~~\bar A = \begin{bmatrix}A &0 \\0& I\end{bmatrix},~~ \bar B = \begin{bmatrix} B\\ 0\end{bmatrix},~~\bar C = \begin{bmatrix} C & 0\end{bmatrix},~~ \bar H = \begin{bmatrix} 0 \\ I\end{bmatrix} \]

This optimal control problem is an instance of $\lone$ robust control~\citep{dahleh19871}.
The optimal controller does not obey the separation principle, and as such, there is no clear notion of an estimated state.

\section{Experimental details} \label{app:experiment-details}

The robust SLS procedure we propose and analyze requires solving a finite dimensional approximation
to an infinite dimensional optimization problem, as $\rowtf$ and the corresponding constraints
\eqref{eq:sls_constraints} and objective function are infinite dimensional objects.
As an approximation, we restrict the system responses $\rowtf$
to be finite impulse response (FIR) transfer matrices of length $T=200$, i.e., we enforce that $\tf \Phi(T) = 0$. We then solve the resulting
optimization problem with MOSEK under an academic license~\citep{mosek}. 
More explicitly, we define $\vop(\tf F) := \rvecthree{F_0^\top}{\dots}{F_{T-1}^\top}^\top$
and $\vbar(\tf F) := \rvecthree{F_0}{\dots}{F_{T-1}}$, where $ F_t$ are the FIR 
coefficients of the system responses. We further define
\begin{equation}
\tf Z := \begin{bmatrix} Q^{1/2} & \\ & R^{1/2} \end{bmatrix}  
\fulltf
\begin{bmatrix}H\\\gamma I \end{bmatrix}.
\end{equation}

The SLS constraints~\eqref{eq:sls_constraints} and FIR condition are then enforced as

\begin{align}
\begin{split}
& \rvectwo{\vbar(\tfxw)}{0} - \rvectwo{0}{A\vbar(\tfxw)} = \rvectwo{0}{B\vbar(\tfuw) + \vbar(\tf I)}\\
& \rvectwo{\vbar(\tfxn)}{0} - \rvectwo{0}{A\vbar(\tfxn)} = \rvectwo{0}{B\vbar(\tfun)}\\
& \cvectwo{\vop(\tfxw)}{0} - \cvectwo{0}{\vop(\tfxw)A} = \cvectwo{0}{\vop(\tfxn)C + \vop(\tf I)}\\
& \cvectwo{\vop(\tfuw)}{0} - \cvectwo{0}{\vop(\tfuw)A} = \cvectwo{0}{\vop(\tfun)C)}\\
& \tfxw(T) = 0, \, \tfuw(T) = 0, \, \tfxn(T) = 0, \, \tfun(T) = 0.
\end{split} \label{eq:actual_constraint}
\end{align}

We then solve the following optimization problem 
\begin{align*}
\underset{\tf \Phi}{\minimize} \quad&\quad \mbox{cost}(\tf Z)\\
\st ~
&~ \eqref{eq:actual_constraint},~~\mbox{norm}(\tfxn) \leq \alpha,
\end{align*}
where the $\mathrm{cost}(\cdot)$ and $\mathrm{norm}(\cdot)$ operators are problem dependent.

For the $\mathcal{L}_1$ robust problem, both the cost function and robust norm constraint reduce to the $\ell_\infty \rightarrow \ell_\infty$
induced matrix norm for an FIR transfer response $\tf F$ with coefficients in $\R^{n \times m}$,
\[
\mbox{norm}_{\infty\rightarrow\infty}(\tf F) = \max_{i=1,\dots,n} \| \vbar(\tf F)_i \|_1,
\]
where $\vbar(\tf F)_i$ denotes the $i$th row of $\vbar(\tf F)$.

For the robust LQG problem, the cost function reduces to the Frobenius norm for an FIR cost transfer matrix $\tf Z$, i.e., 
\[\mathrm{cost}(\tf Z) = \norm{\vop(\tf Z)}_F^2 = \sum_{t=0}^{T} \mathrm{Tr}Z_t^\top Z_t. \]
The robustness constraint is the same $\ell_\infty \rightarrow \ell_\infty$ induced matrix norm as in the $\mathcal{L}_1$ control problem.

 \fi

\end{document}